\documentclass{article}
\usepackage{amssymb}
\usepackage{amsmath}
\usepackage{amsthm}
\usepackage{epsfig}
\usepackage{mathrsfs}
\usepackage{graphicx}

   \usepackage[colorlinks=true]{hyperref}
    \hypersetup{urlcolor=blue, citecolor=red}

\def\RR{\mathbb{R}}

\def\div{\mathrm {div}\,}

\def\vphi{\varphi}
\def\pa{\partial}
\def\na{\nabla}
\def\eps{\varepsilon}

\def\ds{\displaystyle}

\newtheorem{theorem}{Theorem}[section]
\newtheorem{lemma}[theorem]{Lemma}
\newtheorem{proposition}[theorem]{Proposition}

\newtheorem{remark}[theorem]{Remarks}
\newtheorem{rk&ex}[theorem]{Remarks \& Examples}

\def\qed{\hbox{${\vcenter{\vbox{                        %HOLLOW SQUARE
   \hrule height 0.4pt\hbox{\vrule width 0.4pt height 6pt
   \kern5pt\vrule width 0.4pt}\hrule height 0.4pt}}}$}}
   
\title{Fractional diffusion limit for collisional kinetic equations: A moments method}

\author{A. Mellet\thanks{Partially supported by NSF Grant DMS-0901340}}

\date{\small Department of Mathematics, \\
University of Maryland, \\
College Park MD 20742\\
USA}

\begin{document}

\maketitle

\begin{abstract}
This paper is devoted to hydrodynamic limits of linear kinetic equations.
We consider situations in which the thermodynamical equilibrium is described by a heavy-tail distribution function rather than a maxwellian distribution.
A similar problem was addressed in \cite{MMM} using Fourier transform and it was shown that the long time/small mean free path behavior of the solution of the kinetic equation is described by a fractional diffusion equation. 
In this paper, we propose a different method to obtain similar results. This method is somewhat reminiscent of the so-called "moments method" which plays an important role in kinetic theory. 
This new method allows us to consider  space dependent collision operators (which could not be treated in \cite{MMM}).
We believe that it also provides the relevant tool to address nonlinear problems.
\end{abstract}

\paragraph{Mathematics Subject Classification (2000):} 76P05, 35B40, 26A33

\paragraph{Keywords:} Kinetic equations, linear Boltzmann equation, asymptotic analysis, diffusion limit, anomalous diffusion limit, fractional diffusion,  relaxation equation,  anomalous diffusive time scale.

\section{Introduction}
Our goal is to study the asymptotic behavior as $\eps$ goes to zero of the solution $f^\eps(x,v,t)$ of
the following kinetic equation:
\begin{equation}\label{eq:kin0} 
\left\{
\begin{array}{ll}
\theta(\eps) \pa_t f^\eps + \eps v\cdot \na_x f^\eps =L(f^\eps)\quad &x\in\RR^N,\; v\in\RR^N,\; t>0 \\[4pt]
f^\eps(x,v,0) = f_0(x,v)\qquad& x\in\RR^N,\; v\in\RR^N
\end{array}
\right.
\end{equation} 
when the operator $L$  is a  linear relaxation collision operator of the form
\begin{eqnarray}
L(f)&  = &  \int_{\RR^N} \left[ \sigma(x,v,v') f(v') - \sigma(x,v',v) f(v)\right]\, dv'\label{eq:L}
\end{eqnarray}
($\theta(\eps)$, which is such that $\lim_{\eps \to 0}\theta(\eps)=0$ will be chosen later).
As usual, the collision operator is decomposed into a ``gain'' term and a ``loss'' term as follows:
$$L(f)= K(f) -\nu f $$
with
$$ K(f) (x,v) = \int_{\RR^N} \sigma(x,v,v') f(x,v')\, dv'$$
and the collision frequency $\nu$ defined by
$$ \nu(x,v) = \int_{\RR^N} \sigma(x,v',v)\, dv'.$$

This singular perturbation problem is very classical. The underlying goal is the derivation of macroscopic models that describe the evolution of a cloud of particles (represented, at the microscopic level, by the distribution function $f^\eps(x,v,t)$) for  small Knudsen number (of order $\eps$) and large time (of order $\theta(\eps)^{-1}$).
The collision operator (\ref{eq:L}) is one of the simplest operator that  models diffusive, mass-preserving interactions of the particles with the surrounding medium.
We recall  that under reasonable assumptions on the collision kernel $\sigma(x,v,v')$ (see \cite{DGP}), one can show that there exists  a unique equilibrium function $F(x,v)\geq 0$ satisfying 
$$L(F) = 0 \quad \mbox{ and } \quad \int_{\RR^N} F(x,v)\, dv =1 \mbox{ a.e. } x\in\RR^N.$$
Formally, Equation (\ref{eq:kin0}) then leads to 
$$\lim_{\eps\to 0} f^\eps \in \ker(L) = \{\rho(x,t) F(x,v)\, ;\, \rho:\RR^N\times (0,\infty)\longrightarrow \RR\}.$$
Our goal is to show that the density $\rho(x,t)$ is solution to an equation of hydrodynamic type.

\smallskip

The derivation of hydrodynamic limits for kinetic equations such
as (\ref{eq:kin0}) was first investigated by E.~Wigner \cite{W}, A.
Bensoussan, J.L.~Lions and G.~Papanicolaou \cite{BLP} and E.W.~Larsen and J.B.~Keller \cite{LK} and it has been the topic of many papers since (see in particular C. Bardos, R. Santos and  R. Sentis \cite{BSS} and P. Degond, T. Goudon and F. Poupaud \cite{DGP} and references therein). 
In \cite{DGP}, P.~Degond, T.~Goudon
and F.~Poupaud considered very general collision operators of the form
(\ref{eq:L}).  When $F$ decreases "quickly enough" for large values
of $|v|$ (it is often assumed that  $F$ is a Maxwellian distribution function of the form $F(v) = C \exp(\frac{-|v|^2}{2})$),
they proved in particular that for $\theta(\eps) = \eps^2$,
$f^\eps(x,v,t)$ converges, when $\eps$ goes to zero, to a function of
the form $\rho(x,t)F(x,v)$ where the density $\rho(x,t)$ solves a drift-diffusion equation: 
\begin{equation}\label{eq:d0}
 \pa_t\rho - \div_x( D \na_x \rho+U\rho) = 0.
\end{equation}
This result is proved in \cite{DGP} under some assumptions on $F$ that guarantee in particular that the diffusion matrix $D$ and the coefficient $U$, which depends on $F$ and $\sigma$, are finite.

When $F$ is a power tail (or heavy tail) distribution function, satisfying
\begin{equation}\label{eq:tail} 
F(v)\sim\frac{\kappa_0}{|v|^{N+\alpha}} \mbox{ as } |v|\to\infty
\end{equation}
for some $\alpha>0$, the diffusion  matrix $D$ in (\ref{eq:d0}) might however be infinite. 
In  that case, the diffusion limit leading to (\ref{eq:d0}) breaks down, which means that the choice of  time scale $\theta(\eps)=\eps^2$ was inappropriate.
It is the goal of this paper to investigate such situations. 

Power tail distribution functions arise in various contexts, such as astrophysical plasmas (see \cite{ST},  \cite{MR}) or in the study of   granular media (see \cite{EB}, \cite{BG} for the so-called ``inelastic Maxwell model'' introduced in \cite{BCG}). 
We refer to \cite{MMM} for further references concerning the relevance of  power tail distribution functions in various applications.
\smallskip

In \cite{MMM}, S. Mischler, C. Mouhot and the author addressed this problem in the   space homogeneous case (that is with $\sigma$ independent of $x$). It was shown that when $F$ satisfies (\ref{eq:tail}) and the collision frequency $\nu$ satisfies 
$$ \nu(v) \sim \nu_0 |v|^{\beta} \mbox{ as } |v|\to\infty,$$
then for some values of $\alpha$ and $\beta$ ($\alpha>0$ and $\beta<\min(\alpha,2-\alpha)$), the choice of an appropriate time scale $\theta(\eps)$ leads to a fractional diffusion equation instead of~(\ref{eq:d0}).
Let us give a precise statement in the simplest case which corresponds to collision kernels satisfying
$$ \nu_1 F(v) \leq \sigma(v,v')\leq \nu_2 F(v)$$ 
(i.e. $\beta=0$) when (\ref{eq:tail}) holds with $\alpha\in(0,2)$.
Then, taking $\theta(\eps)=\eps^\alpha$, it is shown in \cite{MMM}  that the function $f^\eps(x,v,t)$ converges  to  $\rho(x,t)F(v)$ where $\rho(x,t)$ solves the following fractional diffusion equation:
\begin{equation}\label{eq:fracdiff0}
 \pa_t\rho + \kappa (- \Delta)^{\alpha/2} \rho = 0.
 \end{equation}
\smallskip 

We recall that the fractional power of the Laplacian appearing in (\ref{eq:fracdiff0}) can be defined using the Fourier transform by
$$ \mathcal F((- \Delta)^{\alpha/2} f ) (k) = |k|^\alpha \mathcal F( f)(k).$$
Alternatively, we have the following singular integral representation:
\begin{equation}\label{eq:def} 
(- \Delta)^{\alpha/2} f (x) = c_{N,\alpha} \mbox{PV} \int_{\RR^N} \frac{f(x)-f(y)}{|x-y|^{N+\alpha}}\, dy.
\end{equation}
We refer the reader to Landkof \cite{L} and Stein \cite{S} for a discussion of the properties of fractional operators and singular integrals (we only need to know the definitions above for the purpose of this paper).
\smallskip

Anomalous diffusion limits for kinetic models was first investigated in the case of a gas confined between two plates, when the distance
between the plates goes to $0$ (see \cite{BGTh}, \cite{G}, \cite{D1},
\cite{D2}).  In that case, the limiting equation is still a standard
diffusion equation, but the time scale is anomalous ($\theta(\eps)\sim
\eps^2 \ln(\eps^{-1})$) and 
 the particles travelling in directions nearly parallel to the plates are responsible for the anomalous scaling. 
A fractional diffusion equation such as (\ref{eq:fracdiff0}) was obtained for the first time as a diffusive limit
from a linear phonon-Boltzmann equation  simultaneously by A. Mellet, S. Mischler and C. Mouhot in \cite{MMM} and by M. Jara, T.  Komorowski and S. Olla in \cite{JKO}  (via
a very different probability approach).

\smallskip

The method developed in \cite{MMM} to establish the result quoted above (and more) relies on the use of Fourier transform with respect to $x$, and an explicit computation of the symbol of the asymptotic operator.
Unfortunately, this method is rather complicated to implement when the collision operator depends on the space variable (the Fourier transform of $L(f^\eps)$ involves some convolutions), and virtually impossible to use in a nonlinear framework.
In this paper, we thus propose a different method, which is closer to the so-called ``moment method'' classically used to study hydrodynamic limits of kinetic equations. 
This method relies on the introduction of an appropriate auxiliary problem and a weak formulation of (\ref{eq:kin0}).
More precisely, the corner stone of the proof is to multiply Equation (\ref{eq:kin0}) by a test function $\chi^\eps(x,v,t)$, solution of the auxiliary equation
$$
\nu(x,v) \chi^\eps - \eps v\cdot \na_x \chi^\eps = \nu(x,v) \vphi(x,t)
$$
where $\vphi$ is a smooth test function. 
\smallskip

This method allows us to consider space dependent collision kernels (which could not be treated previously) and we believe  that it  will provide the relevant tools to address non-linear problems (such as equations coupled to Poisson's equation via an electric  field or equations involving non-linear collision operators).

\smallskip

In the next section, we present the main results of this paper, starting with a simple study case which we will use to present the method in a simpler framework (Theorem~\ref{thm:1}). 
We then address the case of general collision operators (Theorem~\ref{thm:2}).
Section \ref{sec:proof1} and \ref{sec:proof2} are devoted to the proof of the main theorems.

\vspace{20pt}
\section{Main results}
\subsection{The simplest case}
The simplest framework corresponds to  collision kernels $\sigma(x,v,v')$ of the form
$$ \sigma(v,v')= F(v)$$
with $F:\RR^N\rightarrow (0,\infty)$ positive function normalized so that
$$ \int_{\RR^N} F(v)\, dv=1. $$
In that case, the collision operator reduces to:
\begin{equation}\label{eq:L0}
 L(f) = \int_{\RR^N} f(v')\, dv' F(v) - f(v),
 \end{equation}
and we have
$$
 \ker(L) = \{\rho(x,t) F(v)\, ;\, \rho:\RR^N\times \RR \to  \RR\}.
$$
We will also assume that the equilibrium distribution function $F(v)$   satisfies
\begin{equation}\label{eq:F0}
 F\in L^\infty(\RR^N), \quad  \quad \mbox{ and }  F(v)=F(-v) \mbox{ for all $v\in\RR^N$} 
\end{equation}
(note that the symmetry assumption is always crucial in deriving diffusion  equations from kinetic models).
\smallskip 

In that case, classical arguments (see for instance \cite{DGP}) show that if we choose $\theta(\eps)=\eps^2$ in (\ref{eq:kin0}), then $f^\eps$ converges as $\eps$ goes to zero to $\rho(x,t)F(v)$ with $\rho$ solution of: 
$$ \pa_t \rho -  \pa_{x_i} (\kappa_{ij}\pa_{x_j} \rho) = 0 $$
where
$$\kappa_{ij}= \int_{\RR^N}v_i v_j F(v)\, dv.$$
However, when $F$ satisfies (\ref{eq:tail}) with $\alpha\in(0,2)$, we  get $\kappa_{ii}=\infty$, so this limit does not make sense. 
Our first result addresses this case:
\begin{theorem}\label{thm:1}
Assume that $L$ is given by (\ref{eq:L0}) with the normalized  distribution function $F(v)> 0$ satisfying (\ref{eq:F0}) and
\begin{equation}\label{eq:Fa} 
|v|^{N+\alpha}\, F(v) \longrightarrow \kappa_0>0 \quad \mbox{ as } |v|\to\infty
\end{equation}
for some  $\alpha\in(0,2)$.
Let $f^\eps(x,v,t)$ be a solution of  (\ref{eq:kin0}) with $\theta(\eps)=\eps^\alpha$ and 
 $f_0\in L^2_{F^{-1}}(\RR^N\times\RR^N)\cap L^1(\RR^N\times\RR^N)$, $f_0\geq 0$.

Then $f^\eps$ converges weakly in
  $L^\infty(0,T;L^2_{F^{-1}}(\RR^N\times\RR^N))$ to $\rho(x,t) F(v)$ with $\rho$ solution of 
\begin{equation}\label{eq:limit0}
\left\{
\begin{array}{l}
\pa_t \rho +   \kappa (-\Delta)^{\alpha/2}  ( \rho )=0\\[4pt]
\rho(x,0)=\rho_0(x)
\end{array}
\right.\end{equation}
where
$$\kappa= \frac{\kappa_0}{c_{N,\alpha}} \int_0^\infty z^\alpha e^{-z}\, dz$$
and $\rho_0=\int_{\RR^N} f_0(x,v)\, dv$ (the constant $c_{N,\alpha}$ is the constant appearing in the definition of the fractional Laplace operator (\ref{eq:def})). 
\end{theorem}

In this result (and in the rest of this paper),  $L^2_{F^{-1}}(\RR^N\times\RR^N)$ denotes the weighted $L^2$ set equipped with the norm:
$$ ||f||_{ L^2_{F^{-1}}(\RR^N\times\RR^N)} =\left( \int_{\RR^N}\int_{\RR^N} \frac{|f(x,v)|^2}{F(v)}\, dv\, dx\right)^{1/2}.$$

Note that Theorem \ref{thm:1} is proved in \cite{MMM} using  Fourier transform. 
We  give a complete proof of this result in Section \ref{sec:proof1}, because
it provides the simplest framework to present our new ``moment'' method. 
In the next section, we obtain new results using this same method.
We point out that the fractional Laplace operator will be obtained here in its singular integral form, rather than its Fourier symbol. This explains that the diffusion coefficient  $\kappa$ is given by very different formulas here and in \cite{MMM}.

\vspace{20pt}

\subsection{General space-dependent collision operators}
We now consider  general collision operators of the form (\ref{eq:L}).
We are mainly interested in situations in which the cross-section $\sigma$ depends on $x$ (this is the case that could not be treated in \cite{MMM}).
Assuming that the equilibrium function $F$ is still independent of $x$, the 
classical diffusion limit, corresponding to the time scale $\theta(\eps)=\eps^2$, leads to
$$ \pa_t \rho - \div_x(D(x)\na_x \rho) = 0 $$
where
$$D(x)=\int_{\RR^n} v\otimes\chi(x,v)F(v)\, dv, \qquad \mbox{ with $\chi$ solution of }\quad  L^*(\chi)=-v$$
($L^*$ denotes the adjoint operator to $L$ in $L^2_v(\RR^N)$).
As before, it is readily seen that for some equilibrium distribution  function $F$ and collision kernel $\sigma$, the matrix $D(x)$ may be infinite. These are the situations that we wish to investigate.
\smallskip

Before stating the result, we need to make the conditions on $\sigma$ and $F$ precise.
The first assumptions are very standard:

\medskip

\noindent {\bf Assumptions (A1)} {\it The cross-section $\sigma(x,v,v')$ 
  is non negative and locally integrable on $\RR^{2N}$ for all $x$.
The collision frequency $\ds \nu(x,v)=\int_{\RR^N} \sigma(x,v',v)\, dv'$ satisfies
$$ 
\nu(x,-v)=\nu(x,v) > 0 \quad \mbox{ for all } x,v\in\RR^N\times \RR^N.
$$
}

\noindent {\bf Assumptions (A2)} {\it There exists a function $ F(v) \in
 L^1(\RR^N)$ \emph{independent of $x$} such that
\begin{equation}\label{def:F}
L(F)=0. 
\end{equation}
Furthermore, the function $F$ is symmetric,  positive and normalized to $1$: 
$$ 
F(-v) =F(v) > 0 \,\, \mbox{ for all }  \,\, v\in \RR^N
 \quad\hbox{and}\quad \int_{\RR^N } F(v)\, dv=1.
$$
}

\medskip

Note that under classical assumptions on $\sigma$, the existence of an
equilibrium function is in fact a consequence of Krein-Rutman's
theorem (see \cite{DGP} for details).  In general, however, this function may depend on $x$ while we assume here that $F$ depends only on $v$ (we discuss in Section \ref{sec:Fx} the case of $x$-dependent equilibrium function).
A particular case in which these assumptions are satisfied is when $\sigma$ is such that 
\begin{eqnarray}\label{def:micro}
\forall \, v,v' \in \RR^N \qquad 
\sigma(x,v,v') =b(x,v,v') F(v), \quad \mbox{ with }  b(x,v',v) = b(x,v,v').
\end{eqnarray} 
In that case, we say that
$\sigma$ satisfies a {\it detailed balanced principle} or a {\it
  micro-reversibility principle}, while the more general assumption
(\ref{def:F}) is called a {\it general balanced principle}.
  
\medskip

The next assumptions concern the  behavior of $F$ and $\nu$ for large
$|v|$.
In order to keep things simple, we will make strong assumptions on $\nu$ and $F$ that makes the limiting process easier (see  Remark \ref{rem:1} below):

\medskip
\noindent {\bf Assumptions (B1)}
{\it 
 There exists $\alpha \in(0,2)$ and a constant $\kappa_0>0$  
such that
\begin{equation}\label{eq:F0def}
|v|^{\alpha+N} F(v) \longrightarrow  \kappa_0  \  \mbox{ as } |v|\rightarrow \infty.
\end{equation}
}

\noindent
{\bf Assumptions (B2)} {\it
There exists $\nu_1$ and $\nu_2$ positive constants such that
$$ \nu_1 F(v) \leq \sigma(x,v,v') \leq \nu_2 F(v).$$
In particular, integrating this condition with  respect to $v$, we deduce:
$$
0<\nu_1\leq \nu(x,v')\leq\nu_2 , \qquad \mbox{ for all } x,v'\in\RR^N\times\RR^N.
$$
Furthermore, we assume that there exists a function $\nu_0(x)$ (satisfying $\nu_1\leq\nu_0(x)\leq\nu_2$) such that
\begin{equation}\label{eq:nu0def}
  \nu(x,v) \longrightarrow \nu_0(x) \ \mbox{ as } |v|\rightarrow \infty.
\end{equation}
uniformly with respect to $x$.
Finally, we assume that $\nu$ is $\mathcal C^1$ with respect to $x$ and 
$$ ||D_x\nu(x,v)||_{L^\infty(\RR^{2N})} \leq C.$$
}

\medskip

We are now ready to state our main result:

\begin{theorem}\label{thm:2}
Assume that (A1), (A2), (B1) and (B2) hold and 
let $f^\eps(x,v,t)$ be a solution of  (\ref{eq:kin0}) with $\theta(\eps)=\eps^\alpha$ and 
 $f_0\in L^2_{F^{-1}}(\RR^N\times\RR^N)\cap L^1(\RR^N\times\RR^N)$, $f_0\geq 0$.

Then $f^\eps$ converges weakly in
  $L^\infty(0,T;L^2_{F^{-1}}(\RR^N\times\RR^N))$ to $\rho(x,t) F(v)$ with $\rho$ solution of 
\begin{equation}\label{eq:limit}
\left\{
\begin{array}{l}
\pa_t \rho +\kappa_0  \mathcal L(\rho )=0\\[4pt]
\rho(x,0)=\rho_0(x)
\end{array}
\right.\end{equation}
where $\rho_0=\int_{\RR^N} f_0(x,v)\, dv$ and $\mathcal L$ is an elliptic  operator of order $\alpha$ defined by the singular integral: 
$$\mathcal L(\rho)=\mbox{PV}\int_{\RR^N} \gamma(x,y)\frac{\rho(x)-\rho(y)}{|x-y|^{N+\alpha}}\, dy$$
with
$$\gamma(x,y) = \nu_0(x)\nu_0(y) \int_0^\infty z^\alpha e^{\ds -z  \int_0^1\nu_0((1-s)x+sy)\, ds}\, dz.$$
\end{theorem}
In view of (B2), it is readily seen that there exist $\gamma_1$ and $\gamma_2$ such that
$$ 0 < \gamma_1 \leq \gamma(x,y) \leq \gamma_2<\infty . $$
In particular the operator $\mathcal L$ has the same order as the fractional Laplace operator $(-\Delta)^{\alpha/2}$.
It is the fractional equivalent of the divergence form elliptic operator $-\div(D\na\, \cdot\,)$ which is typical of conservation laws.
Note also that $\mathcal L$ is self-adjoint since $\gamma(x,y)=\gamma(y,x)$.

\begin{remark}\label{rem:1}
\item[(i)] As in \cite{MMM}, we could replace Assumption (B1) by  $F(v) = F_0(v) \ell(|v|)$ where $F_0$ satisfies (B1) and $\ell(|v|)$ is a slowly varying function (such as a logarithm, power of a logarithm or iterated logarithm).
This would complicate the proof slightly, but can be handled as in \cite{MMM}.
In particular, as in  \cite{MMM}, the slowly varying function $\ell$  does not affect  the asymptotic equation, but the time scale $\theta(\eps)$ would have to  be adjusted ($\theta(\eps)=\eps^\alpha \ell (\eps^{-1})$ instead of $\theta(\eps)=\eps^\alpha$). 
\item[(ii)] Assumption (B2) is far from necessary, and as in \cite{MMM} macroscopic limits can be derived with degenerate collision frequency satisfying $\nu(x,v)\sim\nu_0(x)|v|^\beta$ as $|v|\to \infty$. In that case the operator $\mathcal L$ depends both on the asymptotic behavior of $F$ and that of $\sigma$. 
A tedious but straightforward adaptation of the proof of Theorem~\ref{thm:2} would yield a similar result with  $\alpha'=\frac{\alpha-\beta}{1-\beta}$ instead of $\alpha$ (see \cite{MMM}). 
In particular, we must  take $\theta(\eps)=\eps^{\alpha'}$ and the operator $\mathcal L$ would be replaced by
$$\mathcal L(\rho)=\frac{1}{1-\beta}\mbox{PV}\int_{\RR^N} \gamma(x,y)\frac{\rho(x)-\rho(y)}{|x-y|^{N+\alpha'}}\, dy.$$
\end{remark}

\subsection{Space dependent equilibrium distribution functions and acceleration field} \label{sec:Fx}
In the results above, we have assumed that the equilibrium distribution function $F$ was independent of the space variable $x$ (even though the cross-section may depend on $x$).
However, one might want to consider equilibrium functions depending on $x$ and satisfying
\begin{equation}\label{eq:Fxx}
|v|^{\alpha+N} F(x,v) \longrightarrow  \kappa_0(x)  \  \mbox{ as } |v|\rightarrow \infty
\end{equation}
instead of (\ref{eq:F0def}) with $\alpha\in (0,2)$ and  $\kappa_0(x)> 0$ in $L^\infty(\RR^N)$.
Another natural question concerns the effects of an external acceleration field $E(x,t)$.
In that case, it turns out that the correct scaling is given by
\begin{equation}\label{eq:kin2} 
\!\!\left\{\!\!\!
\begin{array}{ll}
\eps^\alpha \pa_t f^\eps + \eps \, v\cdot \na_x f^\eps  + \eps^{\alpha-1} E \cdot \nabla_v f^\eps =L(f^\eps) &x\in \RR^N, \, v\in\RR^N,\, t>0 \\[4pt]
f^\eps(x,v,0) = f_0(x,v)\qquad& x\in \RR^N, \, v\in\RR^N
\end{array}
\right.
\end{equation} 
and formally, the method described in this paper easily shows that 
  $f^\eps(x,v,t)$  converges 
to $\rho(x,t) F(x,v)$ with $\rho$ solution of 
\begin{equation}\label{eq:limit1}
\left\{
\begin{array}{l}
\pa_t \rho + \mathcal L( \kappa_0\rho)+\div(E\rho)=0\\[4pt]
\rho(x,0)=\rho_0(x)
\end{array}
\right.\end{equation}
where $\mathcal L$ is the elliptic operator defined in Theorem \ref{thm:2}
and $\rho_0=\int_{\RR^N} f_0(x,v)\, dv$.

However, in that case,  the derivation of the crucial $L^2_{F^{-1}}$ estimates (\ref{eq:rho}) and (\ref{eq:gg}) fails. 
Deriving the appropriate estimates (in $L\log L$) to rigorously justify this limit will be the goal of a forthcoming paper.

\vspace{20pt}

\section{Proof of Theorem \ref{thm:1}}\label{sec:proof1}
In this section, we give a detailed proof of  Theorem~\ref{thm:1}. We recall that the collision operator  is given by (\ref{eq:L0})
where  $F(v)> 0$ satisfies the normalization condition $\int F(v)\, dv=1$ and (\ref{eq:F0}).

This proof is organized as follows:
We first recall the derivation of the required a priori estimates. Then, we introduce the auxiliary problem 
(\ref{eq:aux}) and establish the main properties of the solution of that problem. Step 3 is both the simplest and most interesting step: Using the solution of the auxiliary problem in the weak formulation of (\ref{eq:kin0}) we show how the asymptotic problem arises as $\eps\to0$, at least formally. The final step is devoted to the rigorous justification of the limits needed in Step 3.

\medskip

\paragraph{Step 1: A priori estimates for $f^\eps$. } 
 First of all, the maximum principle and the conservation of mass (consequence of the fact that $\int L(f)\, dv=0$ for all $f$) yield:
$$ ||f^\eps(t)||_{L^1(\RR^N\times\RR^N)} = \int_{\RR^N\times\RR^N} f^\eps(x,v,t)\, dx\, dv = ||f_0||_{L^1(\RR^N\times\RR^N)} \quad \mbox{ for all } t>0.$$
Next, we want to prove that $f^\eps$ converges to a function of the form $\rho(x,t) F(v)$.
The following computations are very classical and we only recall them here for the sake of completeness: We write $ f^\eps=\rho^\eps F+g^\eps$ with $\rho^\eps = \int_{\RR^N} f^\eps\, dv$.
Multiplying (\ref{eq:kin0}) by $f^\eps F^{-1}$ and integrating with respect to $x$ and $v$, we get the following equalities:
\begin{eqnarray*}
\eps^\alpha \, \frac{d}{dt} \int_{\RR^{2N}} \frac{(f^\eps)^2}{2} \, F^{-1} \, dvdx 
&=&  \int_{\RR^{2N}} L(f^\eps) \, f^\eps \, F^{-1}  \, dvdx \\
&=&  \int_{\RR^{2N}} [ (\rho^\eps)^2 \, F - (f^\eps)^2 \, F^{-1} ]  \, dvdx \\
&=& -  \int_{\RR^{2N}} [ f^\eps -  \rho^\eps \, F]^2 \, F^{-1}  \, dvdx, 
\end{eqnarray*}
from which we deduce the two estimates
\begin{equation}\label{bdd:thm01}
\sup_{t \ge 0} \int_{\RR^{2N}} \frac{(f^\eps(t,.))^2}{  F}  \, dvdx \le 
 \int_{\RR^{2N}}  \frac{f_0^2}{ F}  \, dv \, dx =  \| f_0 \|^2_{L^2_{F^{-1}}},
\end{equation}
and 
\begin{equation}\label{bdd:thm02}
\int_0^\infty \int_{\RR^{2N}} [ f^\eps -  \rho^\eps \, F]^2 \, F^{-1}
\, dv \, dx \, dt \le   
\frac{\eps^\alpha}2  \,  \| f_0 \|^2_{L^2_{F^{-1}}}.
\end{equation}
We can thus write
$$f^\eps(x,v,t)=\rho^\eps(x,t) F(v) + g^\eps(x,v,t),$$
with
\begin{equation}\label{eq:g} 
||g^\eps||_{L^2_{F^{-1}}} \leq C \eps^{\alpha/2}.
\end{equation}
Cauchy-Schwarz inequality also gives:
$$
\rho^\eps (t,x) = \int_{\RR^N} {f^\eps \over F^{1/2}}Ê\, F^{1/2} \, dv 
\le \left( \int_{\RR^N} {(f^\eps)^2 \over F}Ê\, dv \right)^{1/2},
$$
so that
\begin{equation}\label{bdd:thm03}
\sup_{t \ge 0} \int_{\RR^{N}}  \rho^\eps (t,.)^2 \, dx \le 
  \| f_0 \|^2_{L^2_{F^{-1}}}. 
\end{equation}

We deduce that  $f^\eps$ converges weakly in
  $L^\infty(0,T;L^2_{F^{-1}}(\RR^N\times\RR^N))$ to a function $\rho(x,t) F(x,v)$
where $\rho$ is the weak limit of $\rho^\eps$ in $L^\infty(0,T;L^2(\RR^N))$.

\vspace{20pt}

\paragraph{Step 2: The auxiliary problem. } 
This is the key step of the proof:
For a test function $\vphi\in\mathcal D(\RR^N\times [0,\infty))$, we introduce $\chi^\eps(x,v,t)\in L^\infty_{t,v}((0,\infty)\times\RR^N; L^2_x(\RR^N))$ solution of 
\begin{equation}\label{eq:aux}
\chi^\eps - \eps v\cdot \na_x \chi^\eps = \vphi(x,t).
\end{equation}
This equation is actually easy to integrate, and we check that  $\chi^\eps(x,v,t)$ is  given by the explicit formula:
\begin{equation}\label{eq:chi1}
\chi^\eps(x,v,t) = \int_0^\infty e^{-z} \vphi(x+\eps v z,t)\, dz.
\end{equation}
The function $\chi^\eps$ is thus smooth and bounded in $L^\infty$. 
Furthermore, we have:
\begin{eqnarray*}
\left| \chi^\eps-\vphi\right| &  = & \left|\int_0^\infty e^{-z} \big[ \vphi(x+\eps v z,t)-\vphi(x,t)\big]\, dz\right|\\
& \leq & ||D\vphi||_{L^\infty} \eps |v| ,
\end{eqnarray*}
and thus
$$ \chi^\eps(x,v,t) \longrightarrow \vphi(x,t) \qquad \mbox{ as } \eps\to 0$$
uniformly with respect to  $x$ and $t$.
However, this convergence is not uniform with respect to $v$, so we will need the following lemma:
\begin{lemma}\label{lem:chilimit}
Let $\vphi\in\mathcal D(\RR^N\times [0,\infty))$, and define $\chi^\eps$  by (\ref{eq:chi1}). Then
$$ \int_{\RR^N} F(v) \big[ \chi^\eps(x,v,t) -\vphi(x,t)\big] \, dv \longrightarrow 0 \quad \mbox{ uniformly w.r.t. $x$ and $t$.}
$$
and
$$ \int_{\RR^N} F(v) \big[ \pa_t\chi^\eps(x,v,t) -\pa_t\vphi(x,t)\big] \, dv \longrightarrow 0 \quad \mbox{ uniformly w.r.t. $x$ and $t$.}
$$
Furthermore,
$$ || \chi^\eps||_{L^2_F(\RR^{2N}\times(0,\infty))} \leq  || \vphi||_{L^2((0,\infty)\times\RR^N)}$$
and 
$$ || \pa_t\chi^\eps||_{L^2_F(\RR^{2N}\times(0,\infty))}\leq  ||\pa_t \vphi||_{L^2((0,\infty)\times\RR^N)}.
$$
\end{lemma}
The space  $L^2_F((0,\infty)\times\RR^{2N})$ denotes the weighted $L^2$ space with weight $F(v)$.
 
\begin{proof}[Proof of Lemma \ref{lem:chilimit}]
First of all, we write
$$ \int_{\RR^N} F(v) \big[\chi^\eps(x,v,t) -\vphi(x,t)\big] \, dv  = \int_{\RR^N} F(v) \int_0^\infty e^{-z} \big[\vphi(x+\eps v z,t) -\vphi(x,t)\big]\, dz \, dv.$$
For some large $M$, we then decompose this integral as follows:
\begin{eqnarray*}
&&\left| \int_{\RR^N} F(v) \int_0^\infty e^{-z} \big[\vphi(x+\eps v z,t) -\vphi(x,t)\big]\, dz \, dv \right|\\
&&\qquad\qquad\qquad\quad\quad\quad \leq||D\vphi||_{L^\infty}  \int_{|v|\leq M} F(v) \int_0^\infty e^{-z} \eps |v| z\, dz \, dv\\
&&\qquad\qquad\qquad\quad\quad\quad\qquad\qquad+ 2||\vphi||_{L^\infty} \int_{|v|\geq M} F(v) \int_0^\infty e^{-z} \, dz \, dv\\
&&\quad\qquad\qquad\qquad\quad\quad \leq ||D\vphi||_{L^\infty}  \eps M  +  2||\vphi||_{L^\infty}  \int_{|v|\geq M} F(v) \, dv.
\end{eqnarray*}
Since $F$ is integrable with respect to $v$, it is readily seen that for any $\delta>0$, we can choose $M$ such that $ \int_{|v|\geq M} F(v) <\delta$ and then choose $\eps$ so that $\eps M<\delta$ and deduce that for $\eps$ small enough we have 
$$\left| \int_{\RR^N} F(v) \big[\chi^\eps(x,v,t) -\vphi(x,t)\big] \, dv\right| \leq C\delta$$
where the constant $C$ only depends on $\vphi$.
The first limit in Lemma \ref{lem:chilimit} follows. The other limit can be proved similarly (note that $t$ is only a parameter here).
\medskip

In order to prove the $L^2_F$ estimates, we note that $ \int_0^\infty e^{-z} \, dz=1$ and so 
$$ |\chi^\eps|^2 \leq \int_0^\infty e^{-z} \vphi(x+\eps v z,t)^2\, dz.$$
We deduce: 
\begin{eqnarray*}
|| \chi^\eps||^2_{L^2_F(\RR^{2N}\times(0,\infty))} &= & \int_0^\infty \int_{\RR^{2N}}F |\chi^\eps|^2\, dx\, dv\, dt \\
& \leq & \int_0^\infty\int_{\RR^{2N}} F \int_0^\infty e^{-z} \vphi(x+\eps v z,t)^2\, dz \, dx\, dv\, dt \\
&\leq &  \int_{\RR^{N}}  F(v) \int_0^\infty e^{-z} || \vphi||^2_{L^2(\RR^{N}\times(0,\infty))}\, dz \, dv \\
&\leq &   || \vphi||^2_{L^2(\RR^{N}\times(0,\infty))}.
\end{eqnarray*}
The last inequality is proved similarly.
\end{proof}

\paragraph{Step 3: Weak formulation and formal passage to the limit. } 
Multiplying (\ref{eq:kin0}) by $\chi^\eps$ and integrating with respect to $x,v,t$, we get:
\begin{eqnarray*}
&&\!\!\!\!\!\!\! \!\!\!\!\!\!\!\!\!\! \!\!\!  - \eps^\alpha \int_0^\infty\int_{\RR^{2N}} 
  f^\eps \pa_t \chi^\eps \, dx\, dv\,dt - \eps^\alpha \int _{\RR^{2N}} 
  f_0(x,v) \chi^\eps(x,v,0) \, dx\, dv\\
&&\quad\qquad\qquad = \int_0^\infty\int_{\RR^{2N}} \rho^\eps F \chi^\eps - f^\eps \chi^\eps +f^\eps \eps v\cdot\na_x\chi^\eps\, dx\, dv\,dt
\end{eqnarray*}
which, using (\ref{eq:aux}), yields:
\begin{eqnarray}
&&\!\!\!\!\! \!\!\!\!\!  - \int _0^\infty\int_{\RR^{2N}} 
  f^\eps \pa_t \chi^\eps \, dx\, dv\,dt - \int_{\RR^{2N}}  
  f_0(x,v) \chi^\eps(x,v,0) \, dx\, dv \nonumber \\
 &&\quad\qquad \quad\quad\quad= \eps^{-\alpha} \int_0^\infty\int_{\RR^{2N}}  \rho^\eps F \chi^\eps - f^\eps \vphi(x,t)\, dx\, dv\,dt\nonumber \\
 &&\quad\qquad \quad\quad\quad =  \eps^{-\alpha}  \int_0^\infty\int_{\RR^{2N}}  \rho^\eps F \chi^\eps \, dx\, dv\,dt -  \int \rho^\eps \vphi(x,t)\, dx\,dt. \nonumber
\end{eqnarray}
We deduce:
\begin{eqnarray}
&&\!\!\!\!\! \!\!\!\!\!  - \int _0^\infty\int_{\RR^{2N}} 
  f^\eps \pa_t \chi^\eps \, dx\, dv\,dt - \int_{\RR^{2N}}  
  f_0(x,v) \chi^\eps(x,v,0) \, dx\, dv  \nonumber \\
   &&\quad\quad\quad = \eps^{-\alpha} \int_0^\infty\int_{\RR^{N}}  \rho^\eps \int_{\RR^N} F(v) \left[ \chi^\eps(x,v,t)  - \vphi(x,t)\right]\, dv \,  dx\,dt \nonumber
\end{eqnarray}
or
\begin{eqnarray}
&&\!\!\!\!\! \!\!\!\!\!  - \int _0^\infty\int_{\RR^{2N}} 
  f^\eps \pa_t \chi^\eps \, dx\, dv\,dt - \int_{\RR^{2N}}  
  f_0(x,v) \chi^\eps(x,v,0) \, dx\, dv  \nonumber \\
   &&\qquad\qquad\qquad\qquad\qquad\qquad\qquad =  \int_0^\infty\int_{\RR^{N}}  \rho^\eps  \mathcal L^\eps (\vphi) dx\,dt \label{eq:weak1}
\end{eqnarray}
with
$$ \mathcal L^\eps (\vphi)=\eps^{-\alpha} \int_{\RR^N} F(v) \left[ \chi^\eps(x,v,t)  - \vphi(x,t)\right]\, dv.$$
The rest of the proof of Theorem \ref{thm:1} consists in passing to the limit $\eps\rightarrow 0$ in 
 (\ref{eq:weak1}).
Formally, we immediately see that the left hand side should converge to 
$$ - \int_0^\infty\int_{\RR^{N}}  \rho \, \pa_t \vphi \, dx\,dt  - \int _{\RR^{N}}  
  \rho_0(x) \vphi(x,0) \, dx$$
 (this will actually be a consequence of Lemma \ref{lem:chilimit}).
Passing to the limit in the right hand side of (\ref{eq:weak1}) is  the most interesting part of the proof 
since the nonlocal operator should now appear in the limit of $\mathcal L^\eps$: Using formula (\ref{eq:chi1}) for $\chi^\eps$, we get:
$$
 \mathcal L^\eps (\vphi) = \eps^{-\alpha} \int_{\RR^N}  \int_0^\infty e^{-z}  F(v)[ \vphi(x+\eps v z,t)-\vphi(x,t)]\, dz\, dv.
$$  
The change of variable $w=v\eps z$ and (\ref{eq:Fa}) then yields (formally)
\begin{eqnarray}
\mathcal L^\eps (\vphi)& = &   \eps^{-\alpha} \int_{\RR^N}  \int_0^\infty \frac{1}{|\eps z|^N} e^{-z}  F(w/(\eps z))[ \vphi(x+w,t)-\vphi(x,t)]\, dz\, dw\nonumber \\
 & \sim &  \eps^{-\alpha} \int_{\RR^N}  \int_0^\infty \frac{1}{(\eps z)^N} e^{-z}  \frac{(\eps z)^{N+\alpha}}{|w|^{N+\alpha}} [ \vphi(x+w,t)-\vphi(x,t)]\, dz\, dw\nonumber \\
 & \sim &  \int_{\RR^N}  \int_0^\infty e^{-z}  z^\alpha \frac{1}{|w|^{N+\alpha}} [ \vphi(x+w,t)-\vphi(x,t)]\, dz\, dw\nonumber\\
 & = &- \kappa (-\Delta)^{\alpha/2}\vphi.\label{eq:formal}
\end{eqnarray}  
So we should expect the right hand side of (\ref{eq:weak1}) to converge to 
$$ - \kappa \int_0^\infty\int_{\RR^{N}}  \rho \, (-\Delta)^{\alpha/2}\vphi \,  dx\,dt$$
and we recover the weak formulation of (\ref{eq:limit0}) as the limit of (\ref{eq:weak1}).
\smallskip

Of course, those are just formal computations, so the next and final step consists in rigorously justifying those limits. 
Note that we will have to be particularly careful with the right hand side, since the limit of $\mathcal L^\eps (\vphi)$ involves a singular integral which can only be defined as a principal value (see Proposition~\ref{prop:1}).

\medskip

\paragraph{Step 4: Rigorous passage to the limit in (\ref{eq:weak1}). } 
First, we justify  the  limit in the left hand side of (\ref{eq:weak1}).  We note that
\begin{eqnarray*} 
&& \!\!\!\!\! - \int_0^\infty\int_{\RR^{2N}}  
  f^\eps \pa_t \chi^\eps \, dx\, dv\,dt \\
  && \qquad\qquad =- \int_0^\infty\int_{\RR^{2N}}  
  \rho^\eps F \pa_t \chi^\eps \, dx\, dv\,dt- \int_0^\infty\int_{\RR^{2N}}  
  g^\eps \pa_t \chi^\eps \, dx\, dv\,dt.
\end{eqnarray*}
The second term is obviously bounded by 
$$ || g^\eps||_{L^2_{F^{-1}}((0,\infty)\times\RR^{2N})} || \pa_t\chi^\eps||_{L^2_F((0,\infty)\times\RR^{2N})}
 \leq C \eps^{\alpha/2}
$$
(where we use (\ref{eq:g}) and Lemma \ref{lem:chilimit}). 
The first term can be written as
$$ \int_0^\infty\int_{\RR^{N}} \rho^\eps(x,t) \int_{\RR^N} F(v) \pa_t \chi^\eps \, dv\, dx\,dt
$$
which converges, using Lemma \ref{lem:chilimit}, the weak convergence of $\rho_\eps$ in $L^\infty(0,\infty;L^2(\RR^N))$  and  the fact that $\rho_\eps$ is bounded in $L^\infty(0,\infty;L^1(\RR^n))$, to 
$$ \int_0^\infty\int_{\RR^{N}}  \rho (x,t) \int_{\RR^N} F(v) \pa_t \vphi \, dv\, dx\,dt= \int_0^\infty\int_{\RR^{N}}  \rho (x,t) \pa_t \vphi\, dx\,dt.
$$
Proceeding similarly with the initial data term,  we deduce that the left hand side of (\ref{eq:weak1}) converges to
$$ - \int_0^\infty\int_{\RR^{N}}  \rho \, \pa_t \vphi \, dx\,dt  - \int _{\RR^{N}}  
  \rho_0(x) \vphi(x,0) \, dx.$$
In order to pass to the limit in the right hand side of (\ref{eq:weak1}),  we need the following lemma, which is the rigorous justification of (\ref{eq:formal}):
\begin{proposition}\label{prop:1}
Assume that the conditions of Theorem \ref{thm:1} hold  and that $\chi^\eps$ is defined by (\ref{eq:chi1}). Then
$$ \mathcal L^\eps (\vphi):= \eps^{-\alpha}\int_{\RR^N}  F(v) \left[ \chi^\eps(x,v,t) - \vphi(x,t)\right]\, dv$$
converges as $\eps$ goes to zero to
$$- \kappa (-\Delta)^{\alpha/2} (\vphi) = \kappa\, c_{N,\alpha}\, \mbox{PV}\int_{\RR^N} \frac{\vphi(y)-\vphi(x)}{|y-x|^{N+\alpha}} \, dy  $$
where
$$ \kappa =\frac{\kappa_0}{c_{N,\alpha}} \int_0^\infty    |z|^{\alpha} e^{-z}\, dz.$$
Furthermore, the convergence is uniform with respect to $x$ and $t$.
\end{proposition}
Proposition \ref{prop:1}, the weak convergence of $\rho^\eps$ and the fact that $\rho^\eps$ is bounded in $L^\infty(0,\infty;L^1(\RR^n))$ allow us to pass to the limit in the right hand side of (\ref{eq:weak1}).
We deduce:
$$ - \int \rho \, \pa_t \vphi \, dx\,dt  - \int 
  \rho_0(x) \vphi(x,0) \, dx= - \int \rho   \frac{\kappa}{c_{n,\alpha}} (-\Delta)^{\alpha/2} \vphi\, dx\,dt$$
which is the weak formulation of (\ref{eq:limit0}).
This completes the proof of Theorem \ref{thm:1}, and it only remains to prove Proposition \ref{prop:1}.

\vspace{15pt}

\begin{proof}[Proof of Proposition \ref{prop:1}]
We fix $C>0$ (to be chosen later). Then, using (\ref{eq:chi1}), we write:
\begin{eqnarray*} 
\mathcal L^\eps (\vphi)&= & \eps^{-\alpha} \int_{\RR^N}  F(v) \left[ \chi^\eps (x,v,t) - \vphi(x,t)\right]\, dv \\
&= &  \eps^{-\alpha} \int_{\RR^N}  \int_0^\infty e^{-z}  F(v)[ \vphi(x+\eps v z,t)-\vphi(x,t)]\, dz\, dv\\
&= &  \eps^{-\alpha} \int_{|v|\leq C}  \int_0^\infty e^{-z} F(v)  [ \vphi(x+\eps v z,t)-\vphi(x,t)]\, dz\, dv\\
& &  +
\eps^{-\alpha}\int_{|v|\geq C}  \int_0^\infty e^{-z} \frac{\kappa_0 }{|v|^{N+\alpha}}[ \vphi(x+\eps v z,t)-\vphi(x,t)]\, dz\, dv\\
& &  +
\eps^{-\alpha}\int_{|v|\geq C}  \int_0^\infty e^{-z} \left[ F(v) - \frac{\kappa_0 }{|v|^{N+\alpha}}\right] [ \vphi(x+\eps v z,t)-\vphi(x,t)]\, dz\, dv\\
& =& I_1+I_2+I_3
\end{eqnarray*}
Using the fact that $F(-v)=F(v)$, we can write
\begin{eqnarray}
| I_1| & =  &  \left| \eps^{-\alpha} \int_{|v|\leq C}  \int_0^\infty e^{-z} F(v)  [ \vphi(x+\eps v z,t)-\vphi(x,t)- \eps z v\cdot \na\vphi(x,t)]\, dz\, dv\right| \nonumber \\
& \leq &  ||F||_{L^\infty} \eps^{-\alpha} \int_{|v|\leq C}  \int_0^\infty e^{-z} \left| \vphi(x+\eps v z,t)-\vphi(x,t)- \eps z v\cdot \na\vphi(x,t)\right| \, dz \, dv\nonumber\\
& \leq &  ||F||_{L^\infty} ||D^2\vphi||_{L^\infty} \eps^{-\alpha} \int_{|v|\leq C}  \int_0^\infty e^{-z} |\eps  z v|^2 \, dz \, dv\nonumber\\
& \leq & ||F||_{L^\infty} ||D^2\vphi||_{L^\infty} C^2 \eps^{2-\alpha} \label{eq:I1}
\end{eqnarray}
which goes to zero as $\eps$ goes to zero since $\alpha<2$.

Next, a simple change of variable $w=\eps z v$ gives
\begin{eqnarray}
 I_2 & =  & 
\eps^{-\alpha} \int_0^\infty \int_{|w|\geq C\eps z}  e^{-z} |\eps z|^{N+\alpha} \frac{\kappa_0 }{|w|^{N+\alpha}}[ \vphi(x+w,t)-\vphi(x)]\frac{1}{|\eps z|^N}\, dw\, dz \nonumber\\
& = & \kappa_0  \int_0^\infty    |z|^{\alpha} e^{-z}\int_{|w|\geq C\eps z} \frac{ \vphi(x+w,t)-\vphi(x,t)}{|w|^{N+\alpha}} \, dw\, dz. \label{eq:I2}
\end{eqnarray}
We note that by the very definition of Cauchy Principal Value, we have, for every $z>0$, 
\begin{eqnarray*}
\lim_{\eps\to 0} \int_{|w|\geq C\eps z} \frac{ \vphi(x+w,t)-\vphi(x,t)}{|w|^{N+\alpha}} \, dw & = &   \mbox{PV}\int_{\RR^N} \frac{\vphi(x+w,t)-\vphi(x,t)}{|w|^{N+\alpha}} \, dw\\
& = &-(-\Delta)^{\alpha/2} \vphi.
\end{eqnarray*}
However, this limit is not obviously  uniform with respect to $z$ or $x$, so we have to work a little bit more in order to pass to the limit in (\ref{eq:I2}).
\vspace{10pt}

In order to prove the convergence of $I_2$, we recall that the fractional Laplace operator  can also classically be written as
\begin{eqnarray*}
-(-\Delta)^{\alpha/2} \vphi
& = & 
  \mbox{PV}\int_{\RR^N}[\vphi(x+w)-\vphi(x)]\frac{dw}{|w|^{N+\alpha}} \\
& = & \int_{|w|\geq 1}[\vphi(x+w)-\vphi(x)]\frac{dw}{|w|^{N+\alpha}} \\
&& +\int_{|w|\leq 1}[\vphi(x+w)-\vphi(x)-\na\vphi(x)\cdot w]\frac{dw}{|w|^{N+\alpha}} 
\end{eqnarray*}
where all the integrals are now defined in the usual sense (no principal values).

Using a similar decomposition in (\ref{eq:I2}), we get:
\begin{eqnarray*}
I_2 & =& \kappa_0  \int_0^{1/(C\eps)}   |z|^{\alpha} e^{-z}\int_{|w|\geq 1}[ \vphi(x+w,t)-\vphi(x,t)] \frac{ dw}{|w|^{N+\alpha}}\, dz\\
&&+ \kappa_0  \int_0^{1/(C\eps)}   |z|^{\alpha} e^{-z}\int_{C\eps z\leq |w|\leq 1}[ \vphi(x+w,t)-\vphi(x,t)-\na\vphi(x)\cdot w] \frac{ dw}{|w|^{N+\alpha}}\, dz\\
&& + \kappa_0  \int_{1/(C\eps)}^\infty   |z|^{\alpha} e^{-z}\int_{|w|\geq C\eps z}[ \vphi(x+w,t)-\vphi(x,t)] \frac{ dw}{|w|^{N+\alpha}}\, dz
\end{eqnarray*}
(note that we need to split the $z$ integral to account for the case when $C\eps z>1$).
The last term can be bounded by
\begin{eqnarray*}
&& 2||\vphi||_{L^\infty} \int_{1/(C\eps)}^\infty   |z|^{\alpha} e^{-z}\, dz \int_{|w|\geq 1}\frac{ dw}{|w|^{N+\alpha}}
\end{eqnarray*}
which goes to zero as $\eps$ goes to zero (uniformly with respect to $x$ and $t$),
while the first two terms converge, uniformly with respect to $x$ and $t$, respectively to
\begin{eqnarray*}
&& \kappa_0  \int_0^\infty   |z|^{\alpha} e^{-z}\int_{|w|\geq 1}[ \vphi(x+w,t)-\vphi(x,t)] \frac{ dw}{|w|^{N+\alpha}}\, dz\\
&&\qquad\qquad\qquad\qquad = \kappa\, c_{N,\alpha}\int_{|w|\geq 1}[ \vphi(x+w,t)-\vphi(x,t)] \frac{ dw}{|w|^{N+\alpha}}
\end{eqnarray*}
and
\begin{eqnarray*}
&&\kappa_0  \int_0^\infty   |z|^{\alpha} e^{-z}\int_{ |w|\leq 1}[ \vphi(x+w,t)-\vphi(x,t)-\na\vphi(x)\cdot w] \frac{ dw}{|w|^{N+\alpha}}\, dz.\\
&&\qquad\qquad\qquad\qquad = \kappa\, c_{N,\alpha}\int_{ |w|\leq 1}[ \vphi(x+w,t)-\vphi(x,t)-\na\vphi(x)\cdot w] \frac{ dw}{|w|^{N+\alpha}}.
\end{eqnarray*}
(Note in particular that the integrand in the second term is bounded by $\frac{C ||D^2\vphi||_{L^\infty}}{|w|^{N+\alpha-2}})$ which is integrable at $w=0$).

We deduce:
$$\lim_{\eps\to 0}  I_2(x,t) = -\kappa\, (-\Delta)^{\alpha/2} \vphi(x,t) $$
and the convergence is uniform with respect to $x$ and $t$.
\medskip

Finally, we show that $I_3$ can be made as small as we want: For any $\delta>0$, we can choose $C$ large enough so that
$$\left| F(v) - \frac{\kappa_0 }{|v|^{N+\alpha}}\right| \leq \frac{\delta}{|v|^{N+\alpha}} \mbox{ for all } |v|\geq C.
$$
(note that $C$ was arbitrary up to now).
Proceeding in the same way as with $I_2$, we can thus show:
$$ \limsup_{\eps\to 0 } |I_3|\leq \delta |C(\vphi)|
.$$
Since this holds for any $\delta>0$, the proposition follows.
\end{proof}

\section{Proof of Theorem \ref{thm:2}}\label{sec:proof2}

The proof of Theorem \ref{thm:2} is very similar to that of Theorem \ref{thm:1}, so we will only present in detail the steps that are significantly different.
We recall that $f^\eps$ solves (\ref{eq:kin0}) with $\theta(\eps)=\eps^\alpha$.

We start with recalling the following classical a priori estimates (see Appendix \ref{app:1} for the proof):
\begin{lemma}\label{lem:bound}
  The solution $f^\eps$ of (\ref{eq:kin0}) is bounded in
  $L^\infty(0,\infty, L^1(\RR^{2N}))$ and $L^\infty(0,\infty;L^2_{F^{-1}}(\RR^{2N}))$ uniformly with respect
  to $\eps$.  Furthermore, it satisfies:
  $$ f^\eps = \rho^\eps F(v) + g^\eps, $$
  where the density $ \rho^\eps = \int_{\RR^N} f^\eps\, dv$ and the
  function $g^\eps$ are such that
\begin{equation}\label{eq:rho} 
\|\rho^\eps\|_{L^\infty(0,\infty,L^2(\RR^N) )} \leq \| f_0\|_{L^2_{F^{-1}}} 
\end{equation}
and
\begin{equation}\label{eq:gg} 
\|g^\eps\|_{L^2(0,\infty;L^2_{F^{-1}}(\RR^{2N})} \leq C \, \|f_0\|_{L^2_{F^{-1}}} \,\eps^{\alpha/2}.
\end{equation}

\item In particular  $\rho^\eps$ converges $L^\infty(0,T;L^2)$-weak to $\rho$, and 
$f^\eps$ converges  $L^\infty(0,T;L^2_{F^{-1}})$-weak to  $f=\rho(x,t)F(v)$.
\end{lemma}

Next, we introduce the auxiliary problem corresponding to the general collision operator:
For a test function $\vphi\in\mathcal D(\RR^N\times [0,\infty))$, we  define $\chi^\eps(x,v,t)$ by
\begin{equation}\label{eq:auxnu}
\nu(x,v) \chi^\eps - \eps v\cdot \na_x \chi^\eps = \nu(x,v) \vphi(x,t).
\end{equation}
This equation is slightly more complicated than (\ref{eq:aux}) because of the $x$-dependence of the collision frequency $\nu(x,v)$. However, introducing $\bar\chi(s,x,v,t)= \chi(x+s\eps v,x,v,t)$, a simple computation yields the following formula:
\begin{equation}\label{eq:chi2} 
\chi^\eps(x,v,t) = \int_0^\infty e^{\ds -\int_0^z \nu(x+\eps vs,v)\, ds} \nu(x+\eps v z,v) \vphi(x+\eps v z,t)\, dz 
\end{equation}
(it is relatively easy to check that this function  indeed solves (\ref{eq:auxnu})).
We note that
\begin{equation} \label{eq:unit}
\int_0^\infty e^{\ds -\int_0^z \nu(x+\eps vs,v)\, ds} \nu(x+\eps v z,v) \, dz= \int_0^\infty e^{\ds -u}  \, du  =1
\end{equation}
and so $\chi^\eps(x,v,t)\longrightarrow \vphi$ as $\eps$ goes to zero.
Furthermore, we can prove the following equivalent of Lemma \ref{lem:chilimit} (we recall that $\nu$ is bounded from above and from below by positive constants):
\begin{lemma}\label{lem:chilimit1}
For any $\vphi\in\mathcal D(\RR^N\times [0,\infty))$, if $\chi^\eps$ is defined by (\ref{eq:chi2}), then
$$ \int_{\RR^N} F(v) \big[ \chi^\eps(x,v,t) -\vphi(x,t)\big] \, dv \longrightarrow 0 \quad \mbox{ uniformly w.r.t. $x$ and $t$.}
$$
and
$$ \int_{\RR^N} F(v) \big[ \pa_t\chi^\eps(x,v,t) -\pa_t\vphi(x,t)\big] \, dv \longrightarrow 0 \quad \mbox{ uniformly w.r.t. $x$ and $t$.}
$$
Furthermore,
$$ || \chi^\eps||_{L^2_F((0,\infty)\times\RR^{2N})} \leq C || \vphi||_{L^2((0,\infty)\times\RR^N)}$$
and 
$$ || \pa_t\chi^\eps||_{L^2_F((0,\infty)\times\RR^{2N})}\leq C ||\pa_t \vphi||_{L^2((0,\infty)\times\RR^N)}.
$$
\end{lemma}

\vspace{10pt}

We now have to derive the equivalent of the weak formulation (\ref{eq:weak1}). For that purpose, we 
multiply (\ref{eq:kin0}) by $\chi^\eps$ and integrate with respect to $x,v,t$. We get:
\begin{eqnarray*}
 && \!\!\!\!\! - \eps^\alpha \int_0^\infty \int_{\RR^{2N}} 
  f^\eps \pa_t \chi^\eps \, dx\, dv\,dt -\eps^\alpha \int_{\RR^{2N}} 
  f_0 \chi^\eps(x,v,0) \, dx\, dv\\
&&\qquad \qquad = \int_0^\infty \int_{\RR^{2N}} K(f^\eps) \chi^\eps - \nu f^\eps \chi^\eps +f^\eps \eps v\cdot\na_x\chi^\eps\, dx\, dv\,dt
\end{eqnarray*}
which, using (\ref{eq:auxnu}) and the fact that $K(F)=\nu F$, yields:
\begin{eqnarray}
\!\!\!\!\!&&\!\!\!\!\!\!\!\!\!\!\!\!\!\!\! - \int_0^\infty \int_{\RR^{2N}} 
  f^\eps \pa_t \chi^\eps \, dx\, dv\,dt - \int_{\RR^{2N}} 
  f_0 \chi^\eps(x,v,0) \, dx\, dv\nonumber \\
 &&\!\!\!\!\!\!\!\!\!\! =  \eps^{-\alpha}  \int_0^\infty \int_{\RR^{2N}} f^\eps K^*(\chi^\eps) - f^\eps\nu  \vphi(x,t)\, dx\, dv\,dt\nonumber \\
 & &\!\!\!\!\! \!\!\!\!\! =  \eps^{-\alpha}   \int_0^\infty \int_{\RR^{2N}} f^\eps \left[ K^*(\chi^\eps) - K^*(\vphi)\right]  \, dx\, dv\,dt\nonumber  \\
&&\!\!\!\!\! \!\!\!\!\! =  \eps^{-\alpha}   \int_0^\infty \int_{\RR^{2N}} K(f^\eps) \left[\chi^\eps -\vphi\right]  \, dx\, dv\,dt\nonumber\\
&&\!\!\!\!\! \!\!\!\!\!  =  \eps^{-\alpha}  \int_0^\infty \int_{\RR^{2N}} K(\rho^\eps F) \left[\chi^\eps -\vphi\right]  \, dx\, dv\,dt +  \eps^{-\alpha} \int K(g^\eps) \left[\chi^\eps -\vphi\right]  \, dx\, dv\,dt\nonumber\\
& &\!\!\!\!\! \!\!\!\!\! =  \eps^{-\alpha}  \int_0^\infty \int_{\RR^{2N}} \rho^\eps \nu F \left[\chi^\eps -\vphi\right]  \, dx\, dv\,dt + \eps^{-\alpha}  \int K(g^\eps) \left[\chi^\eps -\vphi\right]  \, dx\, dv\,dt
\label{eq:weak3}
\end{eqnarray}

The proof of Theorem \ref{thm:2} now consists in passing to the limit in (\ref{eq:weak3}). 
As in the proof of Theorem \ref{thm:1}, the left hand side is relatively easy to handle, with the help of Lemma \ref{lem:chilimit1}.  It converges to
$$ - \int_0^\infty \int_{\RR^{N}} 
  \rho \pa_t \vphi \, dx\,dt - \int_{\RR^{2N}} 
  \rho_0 \vphi(x,0) \, dx.
$$
Passing to the limit in the right hand side of (\ref{eq:weak3}) is thus once again the most interesting part of the proof.
We note that there is an additional term involving $g^\eps$ which did not appear in (\ref{eq:weak1}).
This term will be shown to converge to zero as stated in the following lemma:
\begin{lemma}\label{lem:2}
For any test  function $\vphi\in\mathcal D(\RR^N\times [0,\infty))$, let $\chi^\eps$ be defined by (\ref{eq:chi2}). Then
$$\lim_{\eps\rightarrow 0}  \eps^{-\alpha} \int_0^\infty \int_{\RR^{2N}} K(g^\eps) \left[\chi^\eps -\vphi\right]  \, dx\, dv\,dt= 0$$
\end{lemma}
Finally, the key proposition, which is the equivalent of Proposition \ref{prop:1}  is the following:
\begin{proposition}\label{prop:2}
For any test  function $\vphi\in\mathcal D(\RR^N\times [0,\infty))$, let $\chi^\eps$ be defined by (\ref{eq:chi2}). Then
\begin{equation}\label{eq:limit3}
\lim_{\eps\rightarrow 0}
\eps^{-\alpha} \int  \nu(x,v) F(v) \left[\chi^\eps(x,v,t) -\vphi(x,t)\right]  \, dv = -\mathcal L^\star(\vphi) 
\end{equation}
where $\mathcal L$ is defined in Theorem \ref{thm:2}.
Furthermore, this limit is uniform with respect to $x$ and $t$.
\end{proposition}
We leave it to the reader to check that Lemma \ref{lem:2} and Proposition \ref{prop:2} allow us to pass to the limit in (\ref{eq:weak3}) and thus complete the proof of Theorem \ref{thm:2}. We now turn to the proofs of these two results:

\begin{proof}[Proof of Lemma \ref{lem:2}]
First of all, we write (using Assumption (B2)): 
\begin{eqnarray*} 
|K(g)| & \leq &  \int_{\RR^N} \sigma(v,v') |g(v')|\, dv' \\
& \leq &  \nu_2 F(v) \left(\int_{\RR^N} {F(v')}  \, dv'\right) ^{1/2}
\left(\int_{\RR^N}\frac{|g(v')|^2}{F(v')}  \, dv'\right)^{1/2}\\
&  \leq & \nu_2 F \, \|g\|_{L^2_{F^{-1}}(\RR^N)}.
\end{eqnarray*}
for all $x$ and $v$.
Hence
\begin{eqnarray}
&& \left|\int_0^\infty \int_{\RR^{2N}} K(g^\eps) \left[\chi^\eps -\vphi\right]  \, dx\, dv\,dt\right|\nonumber \\
&&\quad\quad \leq\nu_2 
\int_0^\infty \int_{\RR^N} \|g^\eps\|_{L^2_{F^{-1}}(\RR^N)}\left( \int_{\RR^N}  F \left|\chi^\eps -\vphi\right|  \, dv\right)\, dx\,dt\nonumber\\
&&\quad\quad \leq  \nu_2 \|g^\eps\|_{L^2_{F^{-1}}((0,\infty)\times \RR^{2N})}
\left( \int_0^\infty\int _{\RR^{N}}\left( \int_{\RR^{N}} F \left|\chi^\eps -\vphi\right|  \, dv\right)^2\, dx\,dt\right)^{1/2}\nonumber\\
&&\quad\quad \leq C \eps^{\alpha/2} (J_1+J_2)^{1/2}
\label{eq:21}
\end{eqnarray}
with
$$
J_1=\int_0^\infty\int _{\RR^{N}}\left(\int_{|v|\leq \eps^{-1}} F \left|\chi^\eps -\vphi\right|  \, dv\right)^2\, dx\,dt
$$
and 
$$
J_2=\int_0^\infty\int _{\RR^{N}}\left(\int_{|v|\geq \eps^{-1}} F \left|\chi^\eps -\vphi\right|  \, dv\right)^2\, dx\,dt.
$$
Since
\begin{eqnarray*} 
|\chi^\eps -\vphi| & \leq & \int_0^\infty  e^{\ds -\int_0^z \nu(x+\eps vs,v)\, ds} \nu(x+\eps v z,v) |\vphi(x+\eps v z,t) -\vphi(x,t)| \, dz\\
& \leq &||D\vphi||_{L^\infty} \int_0^\infty  e^{\ds -\int_0^z \nu(x+\eps vs,v)\, ds} \nu(x+\eps v z,v) |\eps v z| \, dz\\
& \leq &C ||D\vphi||_{L^\infty}  |\eps v |, 
\end{eqnarray*}
we get (using the fact that $F(v)\leq \frac{C}{|v|^{N+\alpha}}$):
\begin{eqnarray*}
J_1 & \leq &  C||D \vphi||_\infty
 \left( \int_{|v|\leq \eps^{-1}}   F(v)  |\eps v |  \, dv\,\right)^{2} \\
& \leq &  C   \left(C\eps+\int_{1\leq |v|\leq \eps^{-1}} \frac{\eps}{|v|^{N+\alpha-1}}\, dv\right)^{2} \\
&\leq & C(\eps+\eps^{\alpha})^2.
\end{eqnarray*}
To estimate  $J_2$, we first note that (\ref{eq:unit}) implies:
\begin{eqnarray*}
&& \int_0^\infty\int _{\RR^{N}}|\chi^\eps|^2 \,dx\,dt \\
&&  \leq  \int_0^\infty\int _{\RR^{N}}   \int_0^\infty  e^{\ds -\int_0^z \nu(x+\eps vs,v)\, ds} \nu(x+\eps v z,v) |\vphi(x+\eps v z,t)|^2\, dz\,dx\,dt \\
&&  \leq    \int_0^\infty  e^{\ds -\nu_1 z} \nu_2 ||\vphi||_{L^2((0,\infty)\times \RR^N)}^2\, dz\\
&&  \leq\frac{\nu_2}{\nu_1} ||\vphi||_{L^2((0,\infty)\times \RR^N)}^2
\end{eqnarray*}
and so, using the fact that $\int_{|v|\geq \eps^{-1}} F(v)\, dv\leq C\eps^\alpha$, we deduce:
\begin{eqnarray*}
J_2 &\leq &
\left(\int_{|v|\geq \eps^{-1}} F(v)\, dv\right)
\left(\int_0^\infty\int _{\RR^{N}} \int_{|v|\geq \eps^{-1}}  F|\chi^\eps-\vphi|^2  \, dv\, dx\,dt\right)\\
&\leq &C \eps^\alpha
\int_0^\infty\int _{\RR^{N}} \int_{|v|\geq \eps^{-1}}  F[ |\chi^\eps|^2+|\vphi|^2]  \, dv\, dx\,dt\\
&\leq &C \eps^\alpha
 \int_{|v|\geq \eps^{-1}}  F \int_0^\infty\int _{\RR^{N}}|\chi^\eps|^2+|\vphi|^2 \, dx\,dt \, dv\\
&\leq &C \eps^\alpha ||\vphi||^2_{L^2((0,\infty)\times \RR^N)}
 \int_{|v|\geq \eps^{-1}}  F(v)\, dv\\
&\leq &  C \eps^{2\alpha}.
\end{eqnarray*}

Inequality (\ref{eq:21}) thus yields:
$$\eps^{-\alpha}\left|\int_0^\infty \int_{\RR^{2N}} K(g^\eps) \left[\chi^\eps -\vphi\right]  \, dx\, dv\,dt\right|\nonumber \leq C \eps^{-\alpha/2}(\eps+\eps^\alpha) = C (\eps^{\frac{2-\alpha}{2}} + \eps^\frac{\alpha}{2}) $$
and  Lemma \ref{lem:2} follows since $\alpha\in(0,2)$.
\end{proof}

\begin{proof}[Proof of Proposition \ref{prop:2}]
We now turn to the proof of Proposition \ref{prop:2}, which we will only prove under the stronger assumption 
\begin{equation}\label{eq:assstrong}
F(v) = \frac{\kappa_0 }{|v|^{N+\alpha}} \qquad \nu(x,v)=\nu_0(x) \qquad \mbox{for all } |v|\geq C,
\end{equation}
though the result clearly holds in the more general framework (see the proof of Proposition \ref{prop:1}).

Using (\ref{eq:unit}),
we rewrite:
\begin{eqnarray*} 
 &&\!\!\!\!\!\!\!\!\!\!\eps^{-\alpha}  \int_{\RR^N}  \nu(x,v) F(v) \left[\chi^\eps(x,v,t) -\vphi(x,t)\right]  \, dv \\
&&\!\!\!\!\! = 
\eps^{-\alpha} \int_{\RR^N}  \nu(x,v) F(v)
\int_0^\infty e^{ -\int_0^z \nu(x+\eps vs,v)\, ds} \nu(x+\eps v z,v) \Big[\vphi(x+\eps v z,t)-\vphi(x,t)\Big]\, dz \, dv \\
&&\!\!\!\!\!=I_1+I_2
\end{eqnarray*}
where
$$I_1 = \eps^{-\alpha} \int_{|v|\leq C}  \nu(x,v) F(v)
\int_0^\infty e^{ -\int_0^z \nu(x+\eps vs,v)\, ds} \nu(x+\eps v z,v) \Big[\vphi(x+\eps v z,t)-\vphi(x,t)\Big]\, dz \, dv $$
and 
$$I_2 = \eps^{-\alpha} \int_{|v|\geq C}  \nu(x,v) F(v)
\int_0^\infty e^{ -\int_0^z \nu(x+\eps vs,v)\, ds} \nu(x+\eps v z,v) \Big[\vphi(x+\eps v z,t)-\vphi(x,t)\Big]\, dz \, dv .$$

We now proceed as in the proof of Proposition \ref{prop:1} to show that $I_1$ converges to zero, and $I_2$ converges to $ - \mathcal L^\star(\vphi)$.
The only difference with the proof of Proposition \ref{prop:1} is that the factor 
$$p^\eps(x,v,z) =e^{ -\int_0^z \nu(x+\eps vs,v)\, ds} \nu(x+\eps v z,v) $$
is not symmetric with respect to $v$.

To prove the convergence of $I_1$ to zero, we thus introduce 
$$p^0(x,v,z) =e^{ -\int_0^z \nu(x,v)\, ds} \nu(x,v)=e^{ -z \nu(x,v)} \nu(x,v) ,$$
and write
\begin{eqnarray*}
I_1 & = & \eps^{-\alpha} \int_{|v|\leq C} \!\!\!\!\!\! \!\!\!  \nu(x,v) F(v)
\int_0^\infty   p^\eps(x,v,z)      \Big[\vphi(x+\eps v z,t)-\vphi(x,t)\Big]\, dz \, dv\\
& = & \eps^{-\alpha} \int_{|v|\leq C} \!\!\!\!\!\! \!\!\!  \nu(x,v) F(v)
\int_0^\infty   p^0(x,v,z)      \Big[\vphi(x+\eps v z,t)-\vphi(x,t)\Big]\, dz \, dv\\
&&\!\!\!\!\!\! \!\!\! \!\!\!\!\!\!   \!\!\! \!\!\! +\eps^{-\alpha} \int_{|v|\leq C}\!\!\!\!\!\! \!\!\!   \nu(x,v) F(v)
\int_0^\infty   [p^\eps(x,v,z) -   p^0(x,v,z)]  \Big[\vphi(x+\eps v z,t)-\vphi(x,t)\Big]\, dz \, dv
\end{eqnarray*}
Thanks to Assumption (A1) and (A2), we have $\nu(x,-v) F(-v)  p^0(x,-v,z)= \nu(x,v) F(v)  p^0(x,v,z)$, and so proceeding as in (\ref{eq:I1}), we can show that the first term is $\mathcal O(\eps^{2-\alpha})$.
To bound the second term, we note that
$$ |\vphi(x+\eps v z,t)-\vphi(x,t)| \leq C\eps |v| z$$
and 
$$|p^\eps(x,v,z) -   p^0(x,v,z)| \leq C e^{-\nu_1 z} [\eps |v| z + \eps |v| z^2].$$
We deduce
\begin{eqnarray*}
|I_1| & \leq  & C \eps^{2-\alpha}  + C \eps^{-\alpha} \int_{|v|\leq C}  \nu_2 F(v)
\int_0^\infty  C e^{-\nu_1 z} [\eps^2 |v|^2 z^2 + \eps^2 |v|^2 z^3]  \, dz \, dv \\
 & \leq  &C \eps^{2-\alpha} 
\end{eqnarray*}
Where $C$ depends on $||D\vphi||_{L^\infty}$, $||D^2\vphi||_{L^\infty}$,  $||D_x\nu||_{L^\infty}$, $\nu_1$ and $\nu_2$.
We deduce
$$ \lim_{\eps\to 0} I_1 = 0 ,$$
\vspace{10pt}

It only remains to study the limit of $I_2$.
Using condition (\ref{eq:assstrong}), we write:
$$I_2 =\int_0^\infty \eps^{-\alpha} \int_{|v|\geq C}  \frac{\nu_0(x) \kappa_0 }{|v|^{N+\alpha}}
 e^{ -\int_0^z \nu_0(x+\eps vs)\, ds} \nu_0(x+\eps v z) \Big[\vphi(x+\eps v z,t)-\vphi(x,t)\Big]\, dv\, dz  $$
and the change of variable $w=\eps z v $ yields:
\begin{eqnarray*}
I_2 & = & \int_0^\infty \int_{|w|\geq C\eps z} \frac{\nu_0(x) \kappa_0 }{|w|^{N+\alpha}}z^\alpha
 e^{ -\int_0^z \nu_0(x+w\frac{s}{z})\, ds} \nu_0(x+w) \Big[\vphi(x+w,t)-\vphi(x,t)\Big]\, dw\, dz  \\
& =& \int_0^\infty \int_{|w|\geq C\eps z} \frac{\nu_0(x) \kappa_0 }{|w|^{N+\alpha}}z^\alpha
 e^{ -z \int_0^1 \nu_0(x+sw)\, ds} \nu_0(x+w) \Big[\vphi(x+w,t)-\vphi(x,t)\Big]\, dw\, dz. 
\end{eqnarray*}
Formally, this converges to
\begin{eqnarray*}
&& \int_0^\infty \mbox{PV}\int_{\RR^N} \kappa_0  \nu_0(x)\nu_0(x+w)z^\alpha
 e^{ -z \int_0^1 \nu_0(x+sw)\, ds}  \frac{\vphi(x+w,t)-\vphi(x,t)}{|w|^{N+\alpha}}\, dw\, dz\\
&& = - \mathcal L^\star(\vphi)
\end{eqnarray*}
We can now proceed as in the proof of Proposition \ref{prop:1} to rigorously establish this limit.
Note that as above, the factor in the integral is not even with respect to $w$.
We thus introduce
$$
\overline \gamma(x,y,z)=  \kappa_0\, \nu_0(x)\,\nu_0(y)\,z^\alpha\, e^{ \ds - z \int_0^1 \nu_0(x+s(y-x) )\, ds} 
$$
and split $I_2$ as follows:
\begin{eqnarray*}
I_2 & = & \int_0^\infty \int_{|w|\geq C\eps z} \overline \gamma(x,x+w,z) \Big[\vphi(x+w,t)-\vphi(x,t)\Big]\,\frac{ dw}{|w|^{N+\alpha}}\, dz \\
&= & \int_0^{1/(C\eps)} \int_{|w|\geq 1} \overline \gamma(x,x+w,z) \Big[\vphi(x+w,t)-\vphi(x,t)\Big]\,\frac{ dw}{|w|^{N+\alpha}}\, dz \\
&& + \int_0^{1/(C\eps)}\int_{C\eps z \leq |w|\leq 1} \Big[\overline \gamma(x,x+w,z)-\overline \gamma(x,x,z) \Big]\Big[\vphi(x+w,t)-\vphi(x,t)\Big]\,\frac{ dw}{|w|^{N+\alpha}}\, dz \\
&& + \int_0^{1/(C\eps)} \int_{C\eps z \leq |w|\leq 1}\overline  \gamma(x,x,z) \Big[\vphi(x+w,t)-\vphi(x,t)-\na \vphi(x,t)\cdot w \Big]\, \frac{ dw}{|w|^{N+\alpha}}\, dz \\
&& +\int_{1/(C\eps)}^\infty \int_{|w|\geq C\eps z} \overline \gamma(x,x+w,z) \Big[\vphi(x+w,t)-\vphi(x,t)\Big]\, \frac{ dw}{|w|^{N+\alpha}}\, dz 
\end{eqnarray*}
in order to show that 
(\ref{eq:limit3})  holds uniformly with respect to $x$ and $t$ (all the integrals above are defined in the classical sense without need for principal value).
\end{proof}

\bigskip

\appendix

\section{A priori estimates}\label{app:1}
We recall here the proof of Lemma \ref{lem:bound}.
First, we have the following lemma, which  summarizes the key properties of the collision  operator $L$:
\begin{lemma}\label{lem:L}
Assuming that $\sigma$ satisfies Assumption (B2), then the collision operator $ L$ is bounded in $L^2_{F^{-1}}$ and satisfies:
\begin{equation}\label{eq:Q}
  \int _{\RR^N} L(f) \frac{f}{F}\, dv \leq -\nu_1
         \int_{\RR^N} |f-\langle f\rangle F|^2\frac{\nu}{F}\, dv
                      \quad \mbox{ for all } f\in L^2_{F^{-1}}
\end{equation}
where $\langle f\rangle = \int_{\RR^N} f(v)\, dv$.
\end{lemma}

\medskip\noindent{\bf Proof of Lemma~\ref{lem:L}.} We adapt the proof
of \cite[Proposition 1 \& 2]{DGP}.
The fact that $L$ is bounded in $L^2_{F^{-1}}$ is a simple computation, and the proof is left to the reader.
To prove the coercitivity inequality (\ref{eq:Q}), we  write
\begin{eqnarray*}
\int _{\RR^N} L(f) \frac{f}{F}\, dv  & =  & \int _{\RR^N}\int _{\RR^N} \sigma(v,v')f'\frac{f}{F}\, dv\, dv' -\int _{\RR^N} \nu(v)\frac{f^2}{F}\, dv\\
& =  & \int _{\RR^N}\int _{\RR^N} \sigma(v,v')F'\frac{f'}{F'}\frac{f}{F}\, dv\, dv' -\int _{\RR^N} \nu(v)\frac{f^2}{F}\, dv.
\end{eqnarray*}
Next, we note that the second term in the right hand side can be rewritten
\begin{eqnarray*}
\int _{\RR^N} \nu(v)\frac{f^2}{F}\, dv & = & \int _{\RR^N} \int_{\RR^N} \sigma(v',v) F\frac{f^2}{F^2}\, dv\, dv'\\
& = & \int _{\RR^N} \int_{\RR^N} \sigma(v,v') F'\frac{f'^2}{F'^2}\, dv\, dv',
\end{eqnarray*}
as well as  (using the fact that $\nu F=K(F)$)
\begin{eqnarray*}
\int _{\RR^N} \nu(v)\frac{f^2}{F}\, dv & = &\int _{\RR^N} K(F)\frac{f^2}{F^2}\, dv \\
& = & \int _{\RR^N} \int_{\RR^N} \sigma(v,v') F'\frac{f^2}{F^2}\, dv\, dv'.
\end{eqnarray*}

We deduce
\begin{eqnarray}
\int _{\RR^N} L(f) \frac{f}{F}\, dv
& =  & -\frac{1}{2}\int _{\RR^N}\int _{\RR^N} \sigma(v,v')F'  \left[ \frac{f'}{F'}-\frac{f}{F}\right]^2\, dv\, dv' \nonumber \\
& =  & -\frac{1}{2}\int _{\RR^N}\int _{\RR^N} \sigma(v,v')F'  \left[ \frac{g'}{F'}-\frac{g}{F}\right]^2\, dv\, dv' .
\label{coercif1}
\end{eqnarray}
Assumption (B2) then yields
\begin{eqnarray*}\int _{\RR^N} L(f) \frac{f}{F}\, dv & \leq &  -\frac{\nu_1}{2}\int _{\RR^N}\int _{\RR^N} FF'  \left[ \frac{g'}{F'}-\frac{g}{F}\right]^2\, dv\, dv'\\
& =& -\frac{\nu_1}{2}\int _{\RR^N}\int _{\RR^N} F\frac{g'^2}{F'}-2gg'+\frac{g^2}{F}F'\, dv\, dv'
\end{eqnarray*}
Finally, using the fact that $\int g(v)\, dv=0$, we deduce (\ref{eq:Q}).
\qed

\medskip

Using Lemma \ref{lem:L}, we can now prove Lemma \ref{lem:bound}

\medskip\noindent{\bf Proof of Lemma~\ref{lem:bound}.} 
Multiplying (\ref{eq:kin0}) by $f^\eps/F$, we get:
\begin{eqnarray*}
\frac{1}{2} \, \frac{d}{dt} \int_{\RR^{2N}}  |f^\eps|^2\frac{1}{F}\,
dx\,dv  & = &  
\frac{1}{\theta(\eps) }  \, \int_{\RR^{2N}}  L(f^\eps) \frac{f^\eps}{F}\\
& \leq &  - \frac{\nu_1}{\theta(\eps)} \, \int_{\RR^{2N}}\frac{|g^\eps|^2}{F}\, dx\,dv .
\end{eqnarray*}
with $g^\eps=f^\eps-\rho^\eps F$.
We deduce:
\begin{eqnarray*}
&& \!\!\!\!\!\!\!\!\!\!\!\!\!\!\!\!\!\!\!\!\!\!\!\!  \frac12 \, \int_{\RR^{2N}}
\frac{|f^\eps|^2}{F}\, dx\,dv 
  +  \frac{\nu_1}{ \, \theta(\eps)} \, 
  \int_0^t \int_{\RR^{2N}}\frac{|g^\eps|^2 }{F}\, dx\,dv\, ds \\
&& \qquad\qquad\qquad\qquad\qquad\qquad  \leq
\frac12 \, \int_{\RR^{2N}} \frac{|f_0|^2}{F}\, dx\,dv.
\end{eqnarray*}
This inequality shows that $f^\eps$ is bounded in
$L^\infty(0,\infty,L^2_{F^{-1}})$. We also get
$$ 
\int_0^t \int_{\RR^{2N}} \frac{|g^\eps|^2 }{F}\, dx\,dv\, ds \leq 
C \,\|f_0\|_{L^2(F^{-1})} \, \theta(\eps)  .
$$
\medskip
 
Finally, Cauchy-Schwarz inequality implies:
\begin{eqnarray*}
\int_{\RR^N} |\rho^\eps|^2\, dx & = & \int_{\RR^N} \left|\int_{\RR^N} f^\eps\, dv\right|^2 \, dx \\
& \leq & \int_{\RR^{2N}}  |f^\eps|^2\frac{1}{F}\, dv \int_{\RR^N} F\, dv \, dx =
 \int_{\RR^{2N}} |f^\eps|^2\frac{1}{F}\, dv \, dx .
\end{eqnarray*}
\qed

\end{document}